\documentclass[12pt]{article}

\usepackage{graphicx}
\usepackage{latexsym,amssymb}
\usepackage{amsthm}
\usepackage{indentfirst}
\usepackage{amsmath}
\usepackage{color}
\usepackage{xcolor}
\usepackage{fourier}
\usepackage[colorlinks=true,backref=page]{hyperref}
\usepackage[all]{xy}

\newtheorem{theoremalph}{Theorem}

\newtheorem*{Main Theorem}{Main Theorem}

\newtheorem{Theorem}{Theorem}[section]
\newtheorem*{Theorem A}{Theorem A}
\newtheorem*{Theorem A'}{Theorem A'}
\newtheorem*{Theorem B'}{Theorem B'}

\newtheorem{Proposition}[Theorem]{Proposition}
\newtheorem{Lemma}[Theorem]{Lemma}

\newtheorem{Remark-numbered}[Theorem]{Remark}
\newtheorem{Corollary}[Theorem]{Corollary}

\newtheorem*{Claim}{Claim}
\newtheorem{Claim-numbered}[Theorem]{Claim}

 \def\NN{{\mathbb N}} 

\def\QQ{{\mathbb Q}}

 \def\ZZ{{\mathbb Z}}

\def\cE{{\cal E}}

\newcommand{\loc}{\operatorname{loc}}

\def\dim{\operatorname{dim}}

\def\diam{\operatorname{Diam}}

\begin{document}

\title{Ergodic measures with large entropy have long unstable manifolds for $C^\infty$ surface diffeomorphisms}

\author{Chiyi Luo and Dawei Yang\footnote{
D. Yang  was partially supported by National Key R\&D Program of China (2022YFA1005801), NSFC 12171348 and NSFC 12325106.
}}

\date{}
\maketitle

\begin{abstract} 
We prove that for ergodic measures with large entropy have long unstable manifolds for $C^\infty$ surface diffeomorphisms.
Specifically, for any $\alpha>0$, there exist constants $\beta>0$ and $c>0$ such that for every ergodic measure $\mu$ with metric entropy large than $\alpha$, the set of points with the size of unstable manifolds large than $\beta$ has $\mu$-measure large than $c$. 
\medskip

\end{abstract}
\tableofcontents

\section{Introduction}\label{SEC:1}

Pesin theory is one of the cornerstones of smooth ergodic theory, originating with Pesin’s work in the 1970s \cite{Pes76,Pes77}.
A key component of Pesin theory is the result that for a smooth diffeomorphism $f$ of a compact Riemannian manifold $M$ and an ergodic hyperbolic measure $\mu$,  $\mu$-almost every point has a stable manifold and an unstable manifold. 
However, the size of stable manifold or unstable manifold depends on the measure.

Recall the Oseledec theorem \cite{Ose68}: for an ergodic measure $\mu$ of a diffeomorphism $f$, there are finitely many numbers 
$$\lambda_1(\mu,f)>\lambda_2(\mu,f)>\cdots>\lambda_s(\mu,f)$$
and a $Df$-invariant splitting on a full $\mu$-measure set
$$E^1\bigoplus E^2\bigoplus\cdots\bigoplus E^t$$
 such that $\sum_{i=1}^t \dim E^i=d$, and for any non-zero vector $v\in E^i(x)$, 
$$\lim_{n\to\infty}\frac{1}{n}\log\|D_xf^n v\|=\lambda_i(\mu,f),~\mu\text{-a.e.}$$
When there is no confusion, $\lambda_i(\mu,f)$ is simply denoted by $\lambda_i(\mu)$.

Pesin \cite{Pes76,Pes77} has proved that for a $C^{1+\alpha}$ diffeomorphism $f$ and an ergodic measure $\mu$, if $\lambda_1(\mu)>0$, then $\mu$-almost every point has an unstable manifold. 
Take $i_0$ be the maximal integer such that $\lambda_{i_0}(\mu)>0$.
Then for $\mu$-almost every $x$, there exists a Pesin unstable manifold
$$W^u_{\rm Pes}(x):=\left \{y\in M:~\limsup_{n\to\infty}\frac{1}{n}\log d(f^n(x),f^n(y))<0 \right \}$$
which is tangent to $E^1(x)\bigoplus E^2(x)\bigoplus \cdots \bigoplus E^{i_0}(x)$ at $x$. A similar conclusion holds for the stable manifold.

Although the existence of stable and unstable manifolds of a hyperbolic measure is know, the size of these manifolds depends on the measure. 
For $C^\infty$ surface diffeomorphisms, we will prove that the size can be uniformly determined by the metric entropy of $\mu$.

\medskip

We define what it means for the size of unstable manifolds to be “large.” 
Let
$$E^u(x)=E^1(x)\bigoplus E^2(x)\bigoplus \cdots \bigoplus E^{i_0}(x),$$ 
where  $i_0$  is the largest integer such that  $\lambda_{i_0}(\mu) > 0$.
Given $\beta>0$, define
\begin{align*}
&~L^u(\beta)=\{x\in M:~\exists W_x\subset W^u_{\rm Pes}(x),~\textrm{s.t.}~\exp_x^{-1}W_x~\textrm{is a $C^1$ graph of a map~}\varphi:~E^u(x)\to (E^u(x))^\perp,\\
&~~~~{\rm Lip}(\varphi)\le 1/3,~{\rm Domain}(\varphi)\supset E^u(x)(\beta)\}.
\end{align*}
We say that a point $x$ has a $\beta$-large unstable manifold, if $x\in L^u(\beta)$.
For stable manifolds, we let $j_0$ be the smallest integer such that $\lambda_{j_0}(\mu)<0$, and denote
$$E^s(x)=E^{j_0}(x)\bigoplus E^{j_0+1}(x)\bigoplus \cdots \bigoplus E^{t}(x),$$ 
we define
\begin{align*}
&~L^s(\beta)=\{x\in M:~\exists W_x\subset W^s_{\rm Pes}(x),~\textrm{s.t.}~\exp_x^{-1}W_x~\textrm{is a $C^1$  graph of a map~}\varphi:~E^s(x)\to (E^s(x))^\perp,\\
&~~~~{\rm Lip}(\varphi)\le 1/3,~{\rm Domain}(\varphi)\supset E^u(x)(\beta)\}.
\end{align*}
Similarly, we say that a point $x$ has a $\beta$-large stable manifold, if $x\in L^s(\beta)$.

Our main theorem in dimension $2$ (when $M$ is a closed surface: a two-dimensional compact $C^{\infty}$ Riemannian manifold without boundary), states that if the entropy is sufficiently large, both the stable and unstable manifolds are long.

\begin{theoremalph}\label{Thm:C-infty-long}
Suppose $M$ is a closed surface and $f:~M\to M$ is a $C^\infty$ surface diffeomorphism. For any $\alpha>0$, there exist constants $\beta>0$ and $c>0$ such that for any ergodic measure $\mu$, if 
$$h_{\mu}(f)>\alpha,$$ then $\mu(L^u(\beta))>c$ and $\mu(L^s(\beta))>c$.
\end{theoremalph}

{Note that D. Burguet has observed more accurate dependence of $\alpha$ and $c$ in the main theorems.}

{In personal communications with S. Crovisier, Theorem~\ref{Thm:C-infty-long} answers one of Crovisier's questions with an additional assumption: the entropy is uniformly bounded away from $0$. Crovisier has asked that for $C^\infty$ surface diffeomorphisms, if the Lyapunov exponents of an ergodic measure $\mu$ are uniformly bounded away from $0$, then some $\mu$-typical points have long unstable manifolds. See \cite[Section 3, Question 11]{Hen24} for a related question concerning H\'enon maps. Note that by the Ruelle inequality, for surface diffeomorphisms if the entropy is bounded away from $0$, then the Lyapunov exponents are also bounded away from zero.}

The above Theorem~\ref{Thm:C-infty-long} has a version for  $C^r$ diffeomorphisms. Define
$$R(f)=\lim_{n\to\infty}\frac{1}{n}\log\|Df^n\|_{\rm sup}.$$

\begin{theoremalph}\label{Thm:Cr-long-unstable}
Suppose that $M$ is a closed surface and $1<r\in\NN$. 
Let $f:~M\to M$ be a $C^r$ surface diffeomorphism. 
For any $\alpha>0$, there exist constants $\beta>0$ and $c>0$ such that for any ergodic measure $\mu$, if $$h_{\mu}(f)>\alpha+\frac{\max\{R(f),R(f^{-1})\}}{r},$$ then $\mu(L^u(\beta))>c$ and $\mu(L^s(\beta))>c$.
\end{theoremalph}

{For some $C^r$ dissipative surface diffeomorphism, a condition on the Jacobian of the map can imply the existence of long stable manifolds, but not necessarily long unstable manifolds. See \cite{CrP18,Oba21} for instance.}

It is clear that Theorem~\ref{Thm:C-infty-long} follows from Theorem~\ref{Thm:Cr-long-unstable}. These two theorems are possible to get both long stable manifold and long unstable manifold, since in dimension $2$, by applying the Ruelle inequality \cite{Rue78}, every ergodic measure $\mu$ with positive entropy is hyperbolic, i.e., $\lambda_1(\mu)>0>\lambda_2(\mu)$ in the Oseledec theorem.

{Note that Buzzi-Crovisier-Sarig \cite{BCS24} proved the following remarkable result: $C^\infty$ surface diffeomorphisms with positive topological entropy possess the strongly positive recurrence property.
Specifically, there exists a Pesin block such that for each ergodic measure $\mu$ close to the measures of maximal entropy, the Pesin block has a large $\mu$-measure. Consequently, they showed that for ergodic measures close to the measures of maximal entropy, there are points having both long stable manifolds and long unstable manifolds simultaneously. 
With the ``simultenuous'' property, they have very nice properties: the ergodic measures are homoclinically related, and can be coded in one topological Markov shift, etc. 
In our theorem, we can know the existence of long stable manifold and long unstable manifold, but not simultaneously. 
So we do not have other consequences like theirs. 
However, in our main theorems, we do not have to assume the measures are close to the measures of maximal entropy and the size of unstable manifold is uniform with respect to the value of the metric entropy.}

{We also learned that M. Gh\'ezal \cite{Ghezal25}  has proved that for $C^2$ diffeomorphisms, when the metric entropy is larger than some dynamical quantity, some typical points have both long stable manifold and long unstable manifold simultaneously.}

One of the main tools of the proof of Theorem~\ref{Thm:Cr-long-unstable} is based on Yomdin theory \cite{Yom87}, and its recent important progress by Burguet \cite{Bur24,Bur24P}. It mainly uses Yomdin theory for $1$-dimensional curves. Thus, some corresponding versions of Theorem~\ref{Thm:Cr-long-unstable} and Theorem~\ref{Thm:C-infty-long} hold in the higher dimensional case if we deal with the $1$-dimensional entropy.

From the Oseledec theorem stated above, for an ergodic measure $\mu$ {with positive Lyapunov exponents}, one says that $\mu$ has one-dominated Lyapunov exponent, if $\dim E^1=1$ as $E^1$ in the Oseledec theorem. In this case, by the Pesin theory, for $C^{2}$ diffeomorphism $f$, one has one-dimensional strong unstable manifold $W^1(x)$ for $\mu$-almost every point $x$. To describe the complexity of the dynamics along $W^1$, one can define the partial entropy $h^1_\mu(f)$ as in \cite{LeY85}. In the higher-dimensional case, we can prove that if this partial entropy is large, then the size of $W^1$ can be uniformly long. Similar to $L^u(\beta)$, one defines $L^{u,1}(\beta)$:
\begin{align*}
&~L^{u,1}(\beta)=\{x\in M:~\exists W_x\subset W^1(x),~\textrm{s.t.}~\exp_x^{-1}W_x~\textrm{is a $C^1$ graph of a map~}\varphi:~E^1(x)\to (E^1(x))^\perp,\\
&~~~~{\rm Lip}(\varphi)\le 1/3,~{\rm Domain}(\varphi)\supset E^1(x)(\beta)\}.
\end{align*}

\begin{theoremalph}\label{Thm:Cr-long-unstable-high}
Suppose that $M$ is a compact Riemannian manifold without boundary of any dimension. Let $1<r\in\NN$ and $f:~M\to M$ is a $C^r$ diffeomorphism. For any $\alpha>0$, there exist constants $\beta>0$ and $c>0$ such that for any ergodic measure $\mu$ with one dominated Lyapunov exponent, if $$h^1_{\mu}(f)>\alpha+\frac{R(f)}{r},$$ then $\mu(L^{u,1}(\beta))>c$.
\end{theoremalph}

\begin{theoremalph}\label{Thm:C-infty-long-unstable-high}
Suppose that $M$ is a compact Riemannian manifold without boundary of any dimension and $f:~M\to M$ is a $C^\infty$ diffeomorphism. For any $\alpha>0$, there exist $\beta>0$ and $c>0$ such that for any ergodic measure $\mu$ with one dominated Lyapunov exponent, if $$h^1_{\mu}(f)>\alpha,$$ then $\mu(L^{u,1}(\beta))>c$.
\end{theoremalph}
Clearly Theorem~\ref{Thm:C-infty-long-unstable-high} can be obtained as a corollary of Theorem~\ref{Thm:Cr-long-unstable-high}. Theorem~\ref{Thm:Cr-long-unstable-high} can imply Theorem~\ref{Thm:Cr-long-unstable}.
Remark that the measurability of $L^u(\beta)$ and $L^{u,1}(\beta)$ will be checked in Appendix~\ref{Sec:measurability}.

Another interesting subject is the flow $(\varphi^t)_{t\in\mathbb R}$ generated by vector field $X$ over a compact Riemannian manifold $M$ without boundary. In contrast to diffeomorphisms, to study dynamics of flows, usually the difficulties come from the existence of singularities, where the vector field vanishes. See for instance \cite{GaY18,MPP04}. 

However, if we want to establish similar theorems for vector fields, it seems the singularities do not give any trouble because to consider the time-one map of the flow is sufficient to get a result similar to above theorems. For instance, we have the following theorem for three-dimensional vector fields:

\begin{theoremalph}\label{Thm:C-infty-long-vector-field}
Suppose that $M$ is a three-dimensional compact Remiannian manifold without boundary. 
Let $X:~M\to M$ be a $C^\infty$ vector field and $(\varphi^t)_{t\in\mathbb R}$ be the flow generated by $X$. For any $\alpha>0$, there are $\beta>0$ and $c>0$ such that for any ergodic measure $\mu$, if 
$$h_{\mu}(X)=h_{\mu}(\varphi^1)>\alpha,$$ then $\mu(L^u(\beta))>c$ and $\mu(L^s(\beta))>c$.
\end{theoremalph}

Note that Theorem~\ref{Thm:C-infty-long-vector-field} is a consequence of Theorem~\ref{Thm:C-infty-long-unstable-high}. Thus the main work of this paper is to prove Theorem~\ref{Thm:Cr-long-unstable-high}.

\section*{Acknowledgments}
{We are grateful to D. Burguet for his lectures and his comments, and to S. Crovisier for his comments.}


\section{The entropy theory}\label{Sec:entropy}

For a probability measure $\mu$ (not necessarily invariant), for a finite partition $\mathcal P$, define the \emph{static entropy} of $\mu$:
$$H_\mu(\mathcal P)=\sum_{P\in\mathcal P}-\mu(P)\log\mu(P)=\int -\log\mu(P(x)){\rm d}\mu(x).$$
For a diffeomorphism $f$, one defines
$${\mathcal P}^n=\bigvee_{j=0}^{n-1}f^{-j}(\mathcal P).$$
For an invariant measure $\mu$, the metric entropy of $\mu$ with respect to a partition $\mathcal P$ is
$$h_\mu(f,\mathcal P)=\lim_{n\to\infty}\frac{1}{n}H_\mu({\mathcal P}^n);$$
and the \emph{metric entropy} of $\mu$ is defined to be 
$$h_\mu(f)=\sup\{h_\mu(f,\mathcal P): \mathcal{P}~\textrm{is a finite partition of}~M \}.$$
Note that Theorem~\ref{Thm:Cr-long-unstable-high} and Theorem~\ref{Thm:C-infty-long-unstable-high} are results on ``entropy along an invariant foliation'' or ``partial entropies'' as in Ledrappier-Young \cite{LeY85}.

Let $f$ be a $C^r$ diffeomorphism and $\mu$ be an ergodic measure.
Recall the Oseledec theorem as in the introduction, consider $i_0$ such that $\lambda_{i_0}(\mu)>0$. 
As noticed by \cite{Rue79,LeY85}, denote $E=E^1\bigoplus E^2\bigoplus\cdots\bigoplus E^{i_0}$ and define
$$W^E(x):=\left\{y\in M:~\limsup_{n\to\infty}\frac{1}{n}\log d(f^{-n}(x),f^{-n}(y))\le-\lambda_{i_0}\right\}.$$
Then $W^E(x)$ is a $C^r$ $\dim E$-dimensional immersed submanifold of $M$ tangent to $E(x)$ at $x$ for $\mu$-almost every point $x$. 
Each $W^E(x)$ inherits a Riemannian metric from $M$. The distance is denoted by $d^E_{x}$. 
With this distance, one can define $(n,\rho)$-Bowen balls:
$$V^E(x,n,\rho):=\{y\in W^E(x):~d^E_{f^j(x)}(f^j(x),f^j(y))<\rho,~\forall 0\le j<n\}.$$

From \cite{LeS82}, one knows that exists a measurable partition $\xi$ subordinate to $W^E$, i.e., for $\mu$-almost every point $x$, $\xi(x)\subset W^E(x)$ and contains an open neighborhood of $x$ in $W^E(x)$. 
From Rokhlin \cite{Rok67}, associated to each measurable partition $\xi$, there is a system of conditional measures $\{\mu_{\xi(x)}\}$.

As in \cite{LeY85}, the partial entropy along $W^E$ with respect to $\xi$ is defined to be
\begin{equation}\label{eq:09}
	h^E_\mu(x,\xi,f)=\lim_{\rho\to 0}\liminf_{n\to\infty}-\frac{1}{n}\log(\mu_{\xi(x)}(V^E(x,n,\rho)))=\lim_{\rho\to 0}\limsup_{n\to\infty}-\frac{1}{n}\log(\mu_{\xi(x)}(V^E(x,n,\rho))).
\end{equation}
The limits exist for $\mu$-almost every $x$. 
Also noticed in \cite{LeY85}, since $\mu$ is ergodic, $h^E_\mu(x,\xi,f)$ does not depend on $x$; and furthermore, it does not depend on the choice of the measurable partition $\xi$. Hence the partial entropy along $W^E$ for an ergodic measure $\mu$ is defined to be $h^E_\mu(x,\xi,f)$, and is denoted by $h^E_\mu(f)$.

\smallskip

Also noticed in \cite{LeY85}, if $\lambda_{i_0+1}\le 0$ as in the Oseledec theorem, then $h^E_\mu(f)=h_\mu(f)$.
When $E$ is $1$-dimensional, we also denote  $h^1_\mu(f)=h^E_\mu(f)$ as in the statement of Theorem~\ref{Thm:Cr-long-unstable-high} and Theorem~\ref{Thm:C-infty-long-unstable-high}.

\smallskip

Let $x\in M$, $n\in \NN$ and $\rho>0$, the usual Bowen ball is defined by
$$B(x,n,\rho):=\{y\in M:d(f^j(x),f^j(y))< \rho, \ \forall 0\leq j< n\}.$$
It is clear that for $\mu$-almost every $x$ 
$$V^{E}(x,n,\rho)\subset W^{E}_{\rm loc}(x)\cap B(x,n,\rho),$$
and so, we have $\mu_{\xi(x)}(V^E(x,n,\rho))\leq \mu_{\xi(x)}(B(x,n,\rho))$.

\begin{Proposition}\label{Prop:Two Balls}
	For any $\varepsilon>0$ and any $\delta>0$, there exists $K\subset M$ with $\mu(K)>1-\delta$ and $\rho>0$, such that 
	$$\forall x\in K,  \quad h_{\mu}^{E}(f)\leq \liminf_{n\rightarrow +\infty} \frac{\log \mu_{\xi(x)}\left(K\cap B(x,n,\rho)\right)}{-n}+\varepsilon.$$
\end{Proposition}
\begin{proof}
	For every $\varepsilon>0$,  choose $\varepsilon'\in (0,\varepsilon)$, such that for sufficiently large  $n$, one has that
	$$\sum_{j=0}^{\lceil n\varepsilon' \rceil} \tbinom{n}{j}\leq e^{n\varepsilon/3}.$$	
	For every $\delta>0$, choose a subset $F$ with 
	$$\mu(F)>1-\min\left \{\frac{\varepsilon'}{5\|Df\|^{\dim M}_{\sup}},\frac{\delta}{2} \right\}$$ and choose $\rho_{F}>0$ such that for every $0<\rho<\rho_{F}$ and every $x\in F$ the following hold
	\begin{enumerate}
		\item[(1)] $f(V^{E}(f^{-1}(x),1,3\rho))\subset W^E_{\loc}(x)$;
		\item[(2)] $d^{E}_x\left(B(y,\rho)\cap W^E_{\loc}(x)\right)< 3\rho$ for any $y\in M$.
	\end{enumerate}	
	By Equation \eqref{eq:09}, there exists $K\subset F$ with $\mu(K)>1-\delta$, $\rho:=\rho_K\in (0,\rho_{F})$ and $N=N_k\in \NN$ such that 
	\begin{align*}
		 \#\{0\leq k < n:f^k(x)\notin F \}\leq \min\left\{\frac{n\varepsilon}{4\|Df\|_{\sup}^{\dim M}},n\varepsilon'\right\},&~\forall n\geq N,~\forall x\in K; \\
		\mu_{\xi(x)}\left(V^{E}(x,n,6\rho)\right)\leq \exp\left(-n(h^{E}_{\mu}(f)-\dfrac{\varepsilon}{3})\right),&~\forall n\geq N,~\forall x\in K.
	\end{align*}
	Fix $n\geq N$ large enough, let $E(x)=\{0\leq k < n:f^k(x)\notin F\}$, then by the choice of $\varepsilon'$ we have
	$$\# \left\{E(x):x \in K\right\}\leq \sum_{j=0}^{\lceil n\varepsilon' \rceil} \tbinom{n}{j}\leq e^{\frac{n\varepsilon}{3}}.$$
	For a fixed type ${\cal E}\subset [0,n)$, we denote $K_{\cal E}:=\{x\in K,\ E(x)={\cal E}\}$. 
	Then $K$ is covered by at most $e^{\frac{n\varepsilon}{3}}$ disjoint subsets of the form $K_{\cal E}$, and
	$$\forall x\in K, \quad \mu_{\xi(x)}\left(K\cap B(x,n,\rho)\right)\leq e^{\frac{n\varepsilon}{3}} \sup_{\cal E}\mu_{\xi(x)}(K_{\cal E}\cap B(x,n,\rho)).$$
	We now fix a type $\cal E$ such that $ \mu_{\xi(x)}(K_{\cal E}\cap B(x,n,\rho))>0$ and the $\sup$ is obtained.
	\begin{Claim}
		For any $1 \leq k\leq n$, there exist sub-manifolds $\{D_1^k,\cdots D_{c(k)}^k\}$ of $W^{E}_{\loc}(x)$ such that
		\begin{enumerate}
			\item[(a)] $K_{\cal E}\cap B(x,k,\rho)\cap W^{E}_{\loc}(x) \subset \bigcup_{j=1}^{c(k)}D_j^k$;
			\item[(b)] $K_{\cal E}\cap D_j^k\neq \emptyset$, $\diam_{d^{E}_{f^i(x)}}\big (f^i(D_j^k)) \big)\leq 3\rho$ for every $0 \leq i < k$ and every $1\leq j\leq c(k)$;
			\item[(c)] $c(k)\leq \|Df\|_{\sup}^{\#\{0\leq i< k:i\in \cE\}\cdot \dim M}$.
		\end{enumerate}
	\end{Claim}
	\begin{proof}[Proof of the claim]
		For $k=1$, since $K \subset F$ one has $0\notin \cE$. By the choice of $F$, it follows that
		$$K_{\cal E}\cap B(x,\rho) \cap W^{E}_{\loc}(x)\subset  B(x,\rho) \cap W^{E}_{\loc}(x):=D_1^1,$$
		and  $\diam_{d^{E}_{x}}(D_1^1)<3\rho$. 
		This proves the case for $k=1$.
		
		Assume that the claim holds for $k$, then there exist sub-manifolds $\{D_1^k,\cdots D_{c(k)}^k\}$ that satisfy the conclusion of the claim. 
		We now show that the claim also holds for $k+1$. 
		Note that 
		$$B(x,k+1,\rho)=f^{-k}(B(f^k(x),\rho))\bigcap B(x,k,\rho).$$
		It suffices to show that for each $D\in \{D_1^k,\cdots D_{c(k)}^k\}$ with 
		$$f^k(D)\cap f^k(K_{\cE}) \cap B(f^k(x),\rho)\neq \emptyset,$$
		there exist sub-manifolds $\{D(1),\cdots,D(m)\} $ of $D$ such that 
		\begin{enumerate}
			\item [(i)]  $f^{k}(D)\cap B(f^k(x),\rho)\cap f^{k}(K_{\cE}) \subset f^{k} \left(\bigcup_{j=1}^{m} D(j)\right);$
			\item [(ii)]  $D(j)\cap K_{\cE}  \neq \emptyset$ and $\diam_{d_{f^i(x)}^{E}}(f^iD(j))\leq 3\rho, \ \forall 0\leq i\leq k, \ \forall j=1,\cdots,m$;
			\item [(iii)] $m=1, \ k\notin \cE$ or $m\leq \|Df\|_{\sup}^{\dim M}, \ k\in \cE$.
		\end{enumerate}

		If $k\notin {\cal E}$, choose $z\in K_{\cE}\cap D$, then $f^k(z)\in F$. 
		By the choice of $F$, it follows that 
		\begin{align*}
			f^{k}(K_{\cE}) \cap f^k (D)\cap B(f^{k}(x),\rho)\subset W^E_{\loc}(f^{k}(z))\cap B(f^{k}(x),\rho).
		\end{align*}
		Let $D(1)=D\cap f^{-k}(B(f^{k}(x),\rho))$.
		Since $D(1)\subset D$ and $\diam_{d^{E}_{f^k(x)}}(W^E_{\loc}(f^{k}(z))\cap B(f^{k}(x),\rho))\leq 3\rho$,
        we have
		$\diam_{d^{E}_{f^i(x)}}(f^i(D(1)))<3\rho$
	    for every $0\leq i\leq k$.    	
		
		If $k\in {\cal E}$, since $\diam_{d^E_{f^{k-1}(x)}}f^{k-1}(D)\leq 3\rho$, we can decompose $f^k(D)$ into a sequence of sub-manifolds $\{D^k(1),\cdots,D^k(m_1)\}$ such that $m_1\leq \|Df\|_{\sup}^{\dim M}$ and $\diam_{d^{E}_{f^k(x)}}(D^k(i))\leq 3\rho$ for any $i=1,\cdots,m_1$.
		Let $\{D(1),\cdots D(m)\}$ be the set of all $f^{-k}(D^k(j))$ with $1\leq j\leq m_1$ and $D^k(j)\cap B(f^k(x),\rho) \cap f^k(K_{\cE})\neq \emptyset$.  Then $m\leq \|Df\|_{\sup}^{\dim M}$ and
		\begin{align*}
		 f^{k}(D)\cap B(f^k(x),\rho)\cap f^{k}(K_{\cE}) \subset f^{k} \left(\bigcup_{j=1}^{m} D(j)\right),~\diam_{d^{E}_{f^i(x)}}(f^iD(j))<3\rho,~\forall 1\leq i\leq k,~\forall 1\leq j\leq m.
		\end{align*}
		Thus, the claim is proved.
	\end{proof}
	
	By the claim, it follows that 
	$$\mu_{\xi(x)}\left(K_{\cE}\cap B(x,n,\rho)\right)\leq \sum_{j=1}^{c(n)} \mu_{\xi(x)}(D_j^n).$$
	For each $1\leq j\leq c(n)$, choose $x_j^n\in D_j^n \cap K$. 
	Then, we have  $D_j^n\subset V^E(x_j^n,n,6\rho)$.
	By the choice of $\rho$, we have
	\begin{align*}
		\mu_{\xi(x)}\left(K\cap B(x,n,\rho)\right) &\leq e^{\frac{n\varepsilon}{3}} \cdot c(n) \cdot e^{-n(h_{\mu}^{E}(f)-\varepsilon/3)} \\
		&\leq e^{\frac{n\varepsilon}{3}}\cdot e^{(n\varepsilon/3) \frac{\log \|Df\|^{\dim M}_{\sup}}{\|Df\|^{\dim M}_{\sup}} } \cdot e^{-n(h_{\mu}^{E}(f)-\varepsilon/3)} \\
		&\leq  e^{-n(h_{\mu}^{E}(f)-\varepsilon)}.
	\end{align*}
	This completes the proof.
\end{proof}

By Proposition \ref{Prop:Two Balls}, for $\varepsilon>0$ we can choose a compact set $K$ with $\mu(K)>1/2$ and $\rho>0$ such that 
\begin{equation}\label{eq:LimK}
	\forall x\in K,  \quad {h}_{\mu}^{E}(f)\leq \liminf_{n\rightarrow +\infty} -\frac{1}{n}\log \mu_{\xi(x)}\left(K\cap B(x,n,\rho)\right)+\varepsilon.
\end{equation}
Let $x_0\in K, \Sigma \subset W^E_{\loc}(x_0)$ with $\mu_{\xi(x_0)}(K\cap \Sigma)>0$.
\begin{Proposition}\label{Pro:unstable-entropy-partition}
Let $\varepsilon$, $\rho$, $K$ and $\Sigma$ be chosen above,
for any finite partition $\cal P$ satisfying ${\rm Diam}({\cal P})<\rho$ one has
$${h}^E_{\mu}(f)\le\liminf_{n\to\infty}\frac{1}{n}\log\#\{P\cap K\cap \Sigma\neq\emptyset,~P\in{\cal P}^n\}+\varepsilon.$$

\end{Proposition}

\begin{proof}

We consider the probability measure $\mu_{\xi(x_0),K}$ by the following way:
$$\forall~\textrm{Borel set}~A,~~~\mu_{\xi(x_0),K}(A):=\mu_{\xi(x_0),K}^{\Sigma}(A)=\frac{\mu_{\xi(x_0)}(\Sigma\cap K\cap A)}{\mu_{\xi(x_0)}(K\cap \Sigma)}.$$
By the definition of the static entropy $H$, one has
$$H_{\mu_{\xi(x_0),K}}({\cal P}^n)\leq \log\#\{P\cap K\cap \Sigma\neq\emptyset,~P\in{\cal P}^n\}.$$
Hence, it suffices to show that
$${h}^{E}_{\mu}(f)\le\liminf_{n\to\infty}\frac{1}{n}H_{\mu_{\xi(x_0),K}}({\cal P}^n)+\varepsilon.$$
Note that 
 \begin{align*}
	\liminf_{n\rightarrow +\infty} \frac{1}{n} H_{\mu_{\xi(x_0),K}}({\cal P}^n)
	&=   \liminf_{n\rightarrow +\infty} \int -\frac{1}{n} \log  \mu_{\xi(x_0),K}({\cal P}^n(y)) \ {\rm d} \mu_{\xi(x_0),K}(y), \ \ 
	\text{by the definition}     \\
	&\geq \int\liminf_{n\rightarrow +\infty} -\frac{1}{n} \log  \mu_{\xi(x_0),K}({\cal P}^n(y)) \ {\rm d} \mu_{\xi(x_0),K}(y), \ \
	\text{by Fatou's Lemma.}    
\end{align*}
By the definition of $\mu_{\xi(x_0),K}$, one has 
$$\mu_{\xi(x_0),K}({\cal P}^n(y))=\dfrac{\mu_{\xi(x_0)}({\cal P}^n(y) \cap K\cap \Sigma)}{\mu_{\xi(x_0)}(K\cap \Sigma)}.$$
For any $y\in K \cap \xi(x_0)$, we have
$$\mu_{\xi(x_0)}({\cal P}^n(y) \cap K\cap \Sigma)\leq \mu_{\xi(y)}({\cal P}^n(y) \cap K)\leq \mu_{\xi(y)}(B(y,n,\rho) \cap K).$$
By Equation \eqref{eq:LimK} and the measure $\mu_{\xi(x_0),K}$ supported on $K\cap \xi(x_0)$ one has
$$\liminf_{n\rightarrow \infty} -\frac{1}{n} \log  \mu_{\xi(x_0),K}({\cal P}^n(y))\geq {h}^{E}_{\mu}(f)-\varepsilon \quad  \text{for} \  \mu_{\xi(x_0),K}\text{-}a.e. \ y.$$
Therefore, we have 
$$\liminf_{n\rightarrow \infty} \frac{1}{n} H_{\mu_{\xi(x_0),K}}({\cal P}^n)\geq {h}_{\mu}^{E}(f)-\varepsilon.$$
This completes the proof.
\end{proof}

\section{Yomdin-Burguet's reparametrization lemma}

We recall the notion of bounded curves by Burguet \cite{Bur12,Bur24}.
A curve in $M$ is identical to a $C^1$ map $\sigma:~[-1,1]\to M$. Denote by $\sigma_*$ the image of $\sigma$, i.e., $\sigma_*=\sigma([-1,1])\subset M$. Given $L>0$, a $C^r$-curve $\sigma$ is said to be a \emph{$1/L$-curve}, if $$\max_{s=2,\cdots,r}\|D^s\sigma\|_{\rm sup}\le \frac{1}{L}\|D\sigma\|_{\rm sup}.$$ 
Given $L>0$ and $\varepsilon>0$, a $C^r$  $1/L$-curve is said to be $\varepsilon$-strongly bounded, if $\|D\sigma\|_{\rm sup}\le\varepsilon$.
 
 For a linear normed space $A$ and $\delta>0$, we use the notation
 $A(\delta):=\{v\in A:~\|v\|<\delta\}$.
 Since $M$ is a compact manifold, one can choose $r(M)>0$ and $\rho(M)>0$ such that for any $x\in M$, the exponential map $\exp_x$ is injective on $T_x M(r(M))$, and $\exp_x(T_x M(r(M)))\supset B(x,\rho(M))$. 
  
A curve $\sigma$ is said to be \emph{essentially $1/3$-graph}, if there exists $x\in \sigma_*$, a linear subspace $E\subset T_xM$, a subset $E_\sigma\subset E$ containing $0$, and a $C^1$ map $\varphi:~E\to E^\perp$ whose Lipschitz constant is less than $1/3$, such that $\varphi(0)=0$, $\sigma_*\subset B(x,\rho(M))$ and $\exp_x^{-1} \sigma_*=\{(v,\varphi(v)):~v\in E_\sigma\}$. An essential $1/3$-graph is $\beta$-large, if in the above definition, one has that $E_\sigma\supset E(\beta)$.

\begin{Lemma}\label{Lem:curve-graph}
There is $L>0$ such that for any  $1/L$-curve $\sigma$, if ${\rm diam}(\sigma_*)<\rho(M)$ and $\|D\sigma(0)\|<1/L$, then $\sigma$ is an essentially $1/3$-graph.
\end{Lemma}

\begin{proof}
Take $x=\sigma(0)$. Since ${\rm diam}(\sigma_*)<\rho(M)$, one can consider the curve $\exp_x^{-1}\circ\sigma$ in $T_x M$. By taking $E=T_x\sigma_*$ one can represent it as a graph of a map $\psi$ from a subset of $E$ to $E^\perp$.

 By following the proof of \cite[Section 4.1]{Bur24}, one know that the tangent space of every point $\exp_x^{-1}\circ\sigma(t)$ is close to $E$. Thus we know the Lipchitz constant of $\psi$ is less than $1/3$.
\end{proof}

Fix $L$ as in Lemma~\ref{Lem:curve-graph}. Then we simply call $1/L$-curves by bounded curves by following Burguet \cite{Bur24}. Note that in Burguet \cite{Bur24}, $L$ is taken to be $6$. We have to adapt the notion a bit since we want to control the Lipschitz constant precisely.

\begin{Lemma}\label{Lem:size-lower-bounde}
Given $\varepsilon>0$, there is $\beta_\varepsilon>0$ such that if $\sigma$ is a bounded curve and $\|D\sigma(t)\|\ge \varepsilon$ for some $t\in[-1,1]$, then $\sigma$ is $2\beta_\varepsilon$-large.
\end{Lemma}

\begin{proof}
 By following the proof of \cite[Section 5.4]{Bur24}, the lengths of $\sigma([-1,0])$ and $\sigma([0,1])$ are both larger than $\delta$ for some constant $\delta$ related to $\varepsilon$.
After representing it as the graph of a map $\psi$ in the tangent space, by Lemma~\ref{Lem:curve-graph}, the Lipschitz constant of $\psi$ is less than $1/3$. Thus, one knows that the domain of $\psi$ containing $E(\beta)$ for some $\beta>0$ related to $\delta$, hence related to $\varepsilon$.
\end{proof}

A map $\theta:~[-1,1]\to[-1,1]$ is also called a \emph{reparametrization}. A reparametrization $\theta$ is said to be \emph{affine} if $\theta'$ is constant and positive. Note that an affine reparametrization must be contracting, i.e., $\theta'\le 1$.

One has the following reparametrization lemma from Burguet \cite[Lemma12]{Bur24P}. The statement is a bit different, but the proof is essentially the same.
\begin{Lemma}\label{Lem:reparametrization}
Given $r\ge 2$, there is $C_r>0$ with the following property.

For any $C^r$ diffeomorphism $g:~M\to M$, there exists $\varepsilon_g^0>0$ such that for any $\varepsilon\in (0,\varepsilon_g^0)$, for any $\varepsilon$-strongly bounded $C^r$ curve, for any two integers $k$ and $k'$, there is a family $\Theta$ of affine reparametrizations, such that
\begin{enumerate}
\item $\{t\in[-1,1]:~x=\sigma(t),\lceil\log\|Dg(x)\|\rceil=k,~\lceil\log\|Dg|_{T_x{\sigma_*}}\|\rceil=k'\}\subset \bigcup_{\theta\in\Theta}\theta([-1,1])$
\item $g\circ \sigma\circ\theta$ is bounded;
\item $\#\Theta\le C_r{\rm e}^{\frac{k-k'}{r-1}}$.
\end{enumerate}
\end{Lemma}

\begin{Lemma}\label{Lem:more-cutting}
For any $\varepsilon>0$, for any diffeomorphism $g$, if $\sigma$ is $\varepsilon$-strongly bounded, $g\circ\sigma$ is bounded, then there is a family $\Theta_g$ of affine contracting reparametrizations, such that
$$\# \Theta_g\le\|Dg\|_{\rm sup}+2,$$
 and for any $\theta\in\Theta_g$, $g\circ\sigma\circ\theta$ is $\varepsilon$-strongly bounded.
\end{Lemma}
\begin{proof}
One can take a family $\Theta_g$ of affine reparametrizations such that $\#\Theta_g\le \lceil\|Dg\|_{\rm sup}\rceil\le\|Dg\|_{\rm sup}+2$ such that for any $\theta\in\Theta_g$, one has $\theta'\le 1/\|Dg\|_{\rm sup}$ and
$$\bigcup_{\theta\in\Theta_g}\theta([-1,1])=[-1,1].$$
We have to check that  $g\circ\sigma\circ\theta$ is $\varepsilon$-strongly bounded. 
Indeed,
$$\|D(g\circ\sigma\circ\theta)\|\le \|Dg\|_{\rm sup}\cdot\|D\sigma\|\cdot \|D\theta\|\le \|Dg\|_{\rm sup}\cdot \varepsilon \cdot \frac{1}{\|Dg\|_{\rm sup}}\le\varepsilon.$$
This completes the proof of the lemma.
\end{proof}

\section{Choose constants and sets}\label{Sec:const}

Let $\alpha>0$ be as in the statement of Theorem~\ref{Thm:Cr-long-unstable-high}.

\paragraph{Choose $q$.} $q$ is a large integer such that
\begin{itemize}
\item \begin{equation}\label{e.spectral}
\frac{1}{q}\log\|Df^q\|_{\rm sup}-R(f)<\alpha/10.\end{equation}
\item \begin{equation}\label{q-delete-logq}
\frac{1}{q}\log q<\alpha/10.\end{equation}
\item \begin{equation}\label{e.q-alpha-r}
q\alpha>\frac{1}{2r}.\end{equation}
\item \begin{equation}\label{e.q-delete-norm}
\frac{1}{q} \log\big(\max\{\log\|Df\|_{\rm sup},~\log\|Df^{-1}\|_{\rm sup}\}+2\big)<\frac{\alpha}{10}.
\end{equation}
\item For the constant $C_r$ in Lemma~\ref{Lem:reparametrization}.
\begin{equation}\label{e.q-delete-Cr}
\frac{1}{q}\log (2C_r)<\alpha/10.
\end{equation}
\end{itemize}

\paragraph{Choose $\beta$} After $q$ is chosen, for the map $g=f^q$, one has the constant $\varepsilon^0_{f^q}$ from Lemma~\ref{Lem:reparametrization}. 
By taking $\varepsilon_g\in (0,\varepsilon^0_{f^q})$, one assumes that
\begin{equation}\label{e.reduce-epsilon-g}
\|Df\|_{\rm sup}^q\varepsilon_g\le 1
\end{equation}
Choose $\beta>0$ such that
\begin{itemize}
\item $\beta<\beta_{\varepsilon_g}$ as in Lemma~\ref{Lem:size-lower-bounde}. 
\item $$\max_{1\le j\le q}\{\|Df^j\|_{\rm sup}\}\cdot \beta\le\varepsilon_{f^q}.$$
\end{itemize}
\paragraph{Choose $c$.} 

As in Katok's paper \cite[Equation 1.3]{Kat80}, by using Stirling's formula 
\begin{equation*}
	\forall r<\frac{1}{2}, \ \forall k\geq1, \quad \lim_{n\to\infty}\frac{1}{n}\log  \sum_{j=0}^{\lceil nr \rceil} \tbinom{n}{j} \cdot k^j=r\log k-r\log r-(1-r)\log(1-r). 
\end{equation*}
We first choose $\alpha_1>0$ such that there exists $N(\alpha_1)\in \NN$ for which 
\begin{equation}\label{e.choose-alpha-1}
\forall n \geq N(\alpha_1),~\sum_{j=0}^{\lceil n\alpha_1 \rceil} \tbinom{n}{j}\leq e^{n\alpha/10}~\text{and}~\frac{\alpha_1}{q}\log\|Df^q\|_{\rm sup}\leq \alpha/10.
\end{equation}
We then choose $c>0$ by taking
\begin{equation}\label{e.choose-c-alpha1}
c=\frac{\alpha_1}{4q}.
\end{equation}

Assume that $\mu$ is an ergodic measure with one dominated Lyapunov exponent.
We choose a compact set $K:=K_{\mu}$ satisfying $\mu(K)>1/2$ such that the following properties hold:

\begin{enumerate}
	\item The limit measures for points in $K$ converges to $\mu$ uniformly: for any $\varepsilon>0$, there is $N\in\NN$ such that for any $n>N$ and for any $x\in K$, one has that
	$d(\frac{1}{n}\sum_{j=0}^{n-1}\delta_{f^i(x)},\mu)<\varepsilon$.
	
	\item For $n$ large enough, one has that 
	\begin{equation}\label{e.birkhoff-L-u}
	\forall x\in K, \quad \# \{0\leq i<n,f^i(x)\in L^{u,1}(\beta)\}<n (\mu(L^{u,1}(\beta))+c).\end{equation}
	
	\item For $m$ large enough, $$\frac{1}{m}\sum_{j=0}^{m-1}\log\|Df^q|_{E^{u,1}(f^{jq}(y))}\|
\ge \lambda_1(\mu,f^q)-2q\alpha/5$$
	
	\item For $\alpha>0$, there is $\rho>0$ such that for any $x\in K$, one has that
	$$ h_\mu^{1}(f)\le \liminf_{n\to\infty}-\frac{1}{n}\log(\mu_{\xi(x)}(B(x,n,\rho)\cap K))+\alpha/10. $$
\end{enumerate}

We can choose a point $x_0\in K$ such that
\begin{itemize}	
	\item $x_0\in K$, $\mu_{\xi(x_0)}(K)>0$ and for any $\delta>0$ small enough, one has that $\mu_{\xi(x_0)}(K\cap W^u_\delta(x_0))>0$.	
\end{itemize}

\section{Cover $1$-dimensional unstable curves}
Assume that $\mu$ is an ergodic measure with one dominated Lyapunov exponent.
Recall the constants chosen as in Section~\ref{Sec:const}, the compact set $K$ and the point $x_0\in K$ chosen as in Section~\ref{Sec:const}.
We take a reparametrization $\sigma:~[-1,1]\to W^{1}_{\rm loc}(x_0)$ such that $\sigma(0)=x_0$ and $\sigma$ is $\varepsilon_{g}$-bounded. 

\begin{Theorem}\label{Thm:cover}
Assume that $\mu(L^{u,1}(\beta))\le c$.
For any $n\in\NN$, there is a family of affine reparametrizations $\Gamma_n$ having the following properties:
\begin{enumerate}
\item $\limsup_{n\to\infty}\frac{1}{n}\log \#\Gamma_n\le \frac{1}{r} R(f)+\alpha/2$
\item $\bigcup_{\gamma\in\Gamma_n}\sigma\circ\gamma([-1,1])\supset K\cap \sigma_*$;
\item $\|D(f^j\circ\sigma\circ\gamma)\|\le 1,~~\forall j=0,1,\cdots,n-1$.
\end{enumerate}
\end{Theorem}

For the reparametrizations $\{\Gamma_n\}$ in Theorem~\ref{Thm:cover}, for a partition $\mathcal P$ whose diameter is less than $\rho$ as in Proposition~\ref{Pro:unstable-entropy-partition}, we define
$$D(n):=\sup_{\gamma\in \Gamma_n}\#\{P:~P\cap K\cap (\sigma\circ\gamma)_*\neq\emptyset, ~P\in{\cal P}^n\}.$$

Inspired by \cite[Page 1498]{Bur24P}, one has the following proposition.

\begin{Proposition}\label{Pro:sub-exponential-D}Assume  that $\mu(\partial{\cal P})=0$, then we have
$$\lim_{n\to\infty}\frac{1}{n}\log D(n)=0.$$

\end{Proposition}

\begin{proof}
We will prove by contradiction. Assume the conclusion is not true. Then there is $\kappa>0$, such that
\begin{equation}\label{e.D-estimate}
\limsup_{n\to\infty}\frac{1}{n}\log D(n)>4\kappa\log\#{\cal P}.
\end{equation}
Since $\mu(\partial {\mathcal P})=0$, there is $\chi\in(0,1)$ small enough, such that
 \begin{equation}\label{e.boundary-kappa}
 \mu(B(\partial{\cal P},\chi))<\kappa.
 \end{equation}

Equation~\eqref{e.D-estimate} implies that there are sufficiently large $n$, and $\gamma\in\Gamma_n$ such that
\begin{equation}\label{e.cardinatily-partition}
\#\{P:~P\cap K\cap (\sigma\circ\gamma)_*\neq\emptyset, ~P\in{\cal P}^n\}>4([1/\chi]+2)(\#{\cal P})^{2n\kappa}.
\end{equation}
Equation~(\ref{e.boundary-kappa}) implies that for any point $x\in K$, one has that for $n$ large enough,
$$\#\{0\le j\le n-1:~f^j(x)\in B(\partial{\cal P},\chi)\}<n\kappa.$$
Now we take a family of reparametrizations $\Theta$ such that
\begin{itemize}
\item $\#\Theta<4([1/\chi]+2)$;
\item $\big|\theta'\big|<\chi$, for any $\theta\in\Theta$;
\item $\bigcup_{\theta\in\Theta}\theta([-1,1])=[-1,1]$.
\end{itemize}
Thus, there is $\theta\in\Theta$, such that
$$\#\{P:~P\cap K\cap (\sigma\circ\gamma\circ\theta)_*\neq\emptyset, ~P\in{\cal P}^n\}\ge (\#{\cal P})^{2n\kappa}.$$
On the other hand, for any $\theta\in\Theta$, one always has
\begin{align*}
&~~~~~~\#\{P:~P\cap K\cap (\sigma\circ\gamma\circ\theta)_*\neq\emptyset, ~P\in{\cal P}^n\}\\
&\le \prod_{j=0}^{n-1}\#\{P\in{\cal P}:~P\cap f^j(\sigma\circ\gamma\circ\theta([-1,1]))\cap f^j(K)\neq\emptyset\}
\end{align*}
Since $\|D(f^j\circ\sigma\circ\gamma)\|\le 1,~~\forall j=0,1,\cdots,n-1$, one has that $\|D(f^j\circ\sigma\circ\gamma\circ\theta)\|\le \chi/2,~~\forall j=0,1,\cdots,n-1$.
This implies that the diameter of $f^j\circ\sigma\circ\gamma\circ\theta([-1,1])$ is less than $\chi$. Thus,
\begin{itemize}
\item $\#\{P\in{\cal P}:~P\cap f^j(\sigma\circ\gamma\circ\theta([-1,1]))\cap f^j(K)\neq\emptyset\}>1$ if and only if $f^j\circ\sigma\circ\gamma\circ\theta([-1,1])\cap f^j(K)$ contained in $B(\partial{\cal P},\chi)$.
\end{itemize}
Thus, we have 
$$\prod_{j=0}^{n-1}\#\{P\in{\cal P}:~P\cap f^j(\sigma\circ\gamma\circ\theta([-1,1]))\cap f^j(K)\neq\emptyset\}\le \big(\#{\cal P}\big)^{\sup_{x\in K}\#\{j:~f^j(x)\in B(\partial{\cal P},\chi)\}}\le \big(\#{\cal P}\big)^{n\kappa}.$$
This contradicts to Equation~(\ref{e.cardinatily-partition}).

\end{proof}

\section{Proof of Theorem~\ref{Thm:Cr-long-unstable-high}}
For $\alpha>0$ as in the statement of Theorem~\ref{Thm:Cr-long-unstable-high}. We have chosen $\beta>0$ and $c>0$ as in Section~\ref{Sec:const}. Let $\mu$ be an ergodic measure as in the statement of Theorem~\ref{Thm:Cr-long-unstable-high}, i.e., 
$$h^1_{\mu}(f)>\alpha+\frac{R(f)}{r}.$$
We will prove Theorem~\ref{Thm:Cr-long-unstable-high} by absurd, i.e., we assume that $\mu(L^u(\beta))\le c$.

We choose the compact set $K$ and the constant $\rho$ as in Section~\ref{Sec:const}. Choose $x_0\in K$ and a reparametrization $\sigma:~[-1,1]\to W^{1}_{\rm loc}(x_0)$ is $\varepsilon_g$-bounded.
We choose a finite partition $\mathcal P$ such that ${\rm Diam}(\mathcal P)<\rho$ and $\mu(\partial{\mathcal P})=0$.

Recall the family of reparametrizations $\Gamma_n$ in Theorem~\ref{Thm:cover} and $D(n)$ in Proposition~\ref{Pro:sub-exponential-D}.
Note that
$$\#\{P:~P\cap K\cap\sigma_*\neq\emptyset,~P\in{\cal P}^n\}\le \#\Gamma_n\sup_{\gamma\in \Gamma_n}\{P:~P\cap K\cap (\sigma\circ\gamma)_*\neq\emptyset, ~P\in{\cal P}^n\}.$$
By Proposition~\ref{Pro:unstable-entropy-partition}, we have that
$$h_\mu^{1}(f)\le \liminf_{n\to\infty}\frac{1}{n}\log \#\Gamma_n + \lim_{n\to\infty}\frac{1}{n}\log D(n)<\frac{1}{r} R(f)+\alpha$$
by using Theorem~\ref{Thm:cover} and Proposition~\ref{Pro:sub-exponential-D}. Now we have a contradiction to the assumption of Theorem~\ref{Thm:Cr-long-unstable-high}.

\section{The proof of Theorem~\ref{Thm:cover}}

For $n$ in Theorem~\ref{Thm:cover}, we write $n=mq+\ell$, where $\ell\in[0,q)$. In this section, we always take $g=f^q$.
Recall the relationship of $c$ and $\alpha_1$ from Equation~\eqref{e.choose-c-alpha1}.
\begin{Lemma}\label{Lem:density}
For any $x\in K$, for $m$ large enough, one has that
$$\#\{0\le i\le m-1:~f^{qi}(x)\in L^{u,1}(\beta)\}\le\alpha_1 m.$$
\end{Lemma}
\begin{proof}
Since $\mu(L^{u,1}(\beta))< c$, one has that for $m$ large enough, for any $x\in K$, by Equation~\eqref{e.birkhoff-L-u}
$$\#\{0\le j\le mq-1:~f^{j}(x)\in L^{u,1}(\beta)\}\le cmq.$$
Clearly, $\#\{0\le i\le m-1:~f^{qi}(x)\in L^{u,1}(\beta)\}\le \#\{0\le j\le mq-1:~f^{j}(x)\in L^{u,1}(\beta)\}$. Thus, one can conclude by Equation~\eqref{e.choose-c-alpha1}.
\end{proof}

We define the type
$${\cal S}_m=\{E\subset [0,m):~d_m(E)<\alpha_1\}.$$
By the choice of $\alpha_1$ (Equation \eqref{e.choose-alpha-1}),  for $m$ large enough one has that 
$$\frac{1}{m}\log(\sum_{j=0}^{\lceil m\alpha_1 \rceil } \tbinom{m}{j})<\alpha/10,$$
Then, one has that
\begin{equation}\label{e.E-type}
\#{\cal S}_m\le {\rm e}^{\alpha n/10}.
\end{equation}

For any  $E\in {\cal S}_m$, we consider
$$K_E=\{x\in K\cap\sigma_*:~f^{jq}(x)\in L^{u,1}(\beta) ~{\textrm {iff}}~j\in E\}.$$
Given $2m$-integers $(k_0,k_0',k_1,k_1',\cdots,k_{m-1},k_{m-1}')$, we define a subset of $K_E$:
$$\Sigma(k_0,k_0',k_1,k_1',\cdots,k_{m-1},k_{m-1}')=\Sigma(E;k_0,k_0',k_1,k_1',\cdots,k_{m-1},k_{m-1}')$$
by the following way: $x\in \Sigma(k_0,k_0',k_1,k_1',\cdots,k_{m-1},k_{m-1}')$ if and only if $x\in K_E$ and
$$\forall 0\le j\le m-1,~~~\lceil\log\|Dg(g^j(x))\|\rceil=k_j,~\lceil\log\|Dg|_{T_{g^j(x)}{(g^j(\sigma))_*}}\|\rceil=k_j'$$
If $\Sigma(k_0,k_0',k_1,k_1',\cdots,k_{m-1},k_{m-1}')\neq\emptyset$, then the $2m$-tuple $(k_0,k_0',k_1,k_1',\cdots,k_{m-1},k_{m-1}')$ is said to be \emph{admissible}.

\smallskip

Meanwhile, for a point $y\in \sigma_*$, for any $j\ge 0$, we define functions:
$$k_j(y):=\lceil\log\|Dg(g^j(y))\|\rceil,~~~k_j'(y):=\lceil\log\|Dg|_{T_{g^j(y)}{(g^j(\sigma))_*}}\|\rceil.$$
\begin{Lemma}\label{Lem:derivetive-type}
The are at most
$$\big(q\cdot \max\{\log\|Df\|_{\rm sup}+2,~\log\|Df^{-1}\|_{\rm sup}+2\}\big)^{2m}$$
admissible $2m$-tuples.
\end{Lemma}
\begin{proof}
For any $j$, $k_j$ or $k_j'$ has $\max\{\lceil\log\|Df^q\|_{\rm sup}\rceil,\lceil\log\|Df^{-q}\|_{\rm sup}\rceil\}+1$ possibilities. It is less than $q\cdot \max\{\log\|Df\|_{\rm sup}+2,~\log\|Df^{-1}\|_{\rm sup}+2\}$. Since we have $2m$ integers, one gets the estimates.
\end{proof}

\begin{Lemma}\label{Lem:sum-of-k}
For $n=mq$ large enough, for $y\in K$, one has that
$$\frac{1}{n}\frac{1}{r-1}\sum_{j=0}^{m-1}(k_j(y)-k_j'(y))\le \frac{1}{r}\frac{1}{q} \log \|Df^q\|_{\rm sup}\le \frac{1}{r}R(f)+\alpha/10$$
\end{Lemma}
\begin{proof}
By our assumption:
$$h^1_{\mu}(f)>\alpha+\frac{R(f)}{r},$$
Thus, by Equation~\eqref{e.spectral},
$$h^1_\mu(f^q)=q h^1_\mu(f)\ge q(\frac{R(f)}{r}+\alpha)\ge q\left(\frac{1}{r}(\frac{\log \|Df^q\|_{\rm sup}}{q}-\frac{\alpha}{10})+\alpha\right)\ge \frac{\log\|Df^q\|_{\rm sup}}{r}+9q\alpha/10.$$
By the Ruelle inequality \cite{Rue78} (see \cite[Section 10.1]{LeY85} for the version of partial entropy), one has that $\lambda_1(\mu,f^q)\ge h^1_\mu(f^q)>\frac{1}{r}\log \|Df^q\|_{\rm sup}+9q\alpha/10$.
Thus, for $m$ large enough, for any $y\in K$,
\begin{align*}
&\frac{1}{m}\sum_{j=0}^{m-1}\log\|Df^q|_{E^{u,1}(f^{jq}(y))}\|-\frac{1}{r}\frac{1}{m}\sum_{j=0}^{m-1}\log \|Df^q(f^{jq}(y))\|\\
&\ge \lambda_1(\mu,f^q)-2q\alpha/5-\frac{1}{r}\log \|Df^q\|_{\rm sup}>q\cdot\alpha/2.
\end{align*}
Together with Equation~\eqref{e.q-alpha-r}, this implies that 
$$\sum_{j=0}^{m-1}k_j'(y)-\frac{1}{r}\sum_{j=0}^{m-1}k_j(y)>m(q\cdot\alpha/2-1/r)>0,$$
in other words,
$$\sum_{j=0}^{m-1}k_j'(y)> \frac{1}{r}\sum_{j=0}^{m-1}k_j(y).$$
Thus, one has
$$\frac{1}{r-1}\sum_{j=0}^{m-1}(k_j(y)-k_j'(y))\le\frac{1}{r-1}\frac{r-1}{r}\sum_{j=0}^{m-1}k_j(y)\le \frac{1}{r}m \log \|Df^q\|_{\rm sup}.$$
By dividing $n=mq$, by Equation~\eqref{e.spectral}, one has that
$$\frac{1}{n}\frac{1}{r-1}\sum_{j=0}^{m-1}(k_j(y)-k_j'(y))\le \frac{1}{r}\frac{1}{q} \log \|Df^q\|_{\rm sup}\le \frac{1}{r}R(f)+\alpha/10.$$
This completes the proof of the Lemma.
\end{proof}

We will prove the following proposition.

\begin{Proposition}\label{Pro:one-type-covering}
For the set $\Sigma(k_0,k_0',k_1,k_1',\cdots,k_{j-1},k_{j-1}'),~1\le j\le m$, there is a family $\Gamma_j'=\Gamma_j'(\Sigma(k_0,k_0',k_1,k_1',\cdots,k_{j-1},k_{j-1}'))$ of reparametrizations satisfying
$$\#\Gamma_j'\le 2^j C_r^j\exp\{\sum_{i=0}^{j-1}\frac{k_i-k_i'}{r-1}\}(\|Dg\|_{\rm sup}+2)^{\#\{0< i\le j,~i\in E\}},$$
and the following properties:
\begin{enumerate}
\item $\Sigma(k_0,k_0',k_1,k_1',\cdots,k_{j-1},k_{j-1}')\subset\bigcup_{\gamma\in\Gamma_j'}\sigma\circ\gamma([-1,1])$
\item for any $\gamma\in\Gamma_j'$, $g^i\circ\sigma\circ\gamma$ is $\varepsilon_g$-bounded for any $1\le i\le j$.
\end{enumerate}
\end{Proposition}
\begin{proof}
We will prove this proposition by induction. We first consider the case $j=1$. In this case, we consider the set $\Sigma(k_0,k_0')$. By applying the reparametrization Lemma (Lemma~\ref{Lem:reparametrization}), there is a family $\Theta$ of affine reparametrizations, whose cardinality satisfying
$$\#\Theta\le C_r{\rm e}^{\frac{k_0-k_0'}{r-1}},$$
such that
\begin{enumerate}
\item $\Sigma(k_0,k_0')\subset \bigcup_{\theta\in\Theta}\sigma\circ \theta([-1,1])$
\item $g\circ \sigma\circ\theta$ is bounded for any $\theta\in\Theta$;
\end{enumerate}
Note that $g\circ\sigma\circ\theta$ may not be $\varepsilon_g$-bounded. However, one notices that $\sigma\circ\theta$ is $\varepsilon_g$-bounded.

We have two cases
\begin{itemize}
\item $1\in E$: this means $g(x)=f^q(x)\in L^{u,1}(\beta)$ for all $x\in \Sigma(k_0,k_0')$.
\item $1\notin E$: this means $g(x)=f^q(x)\notin L^{u,1}(\beta)$ for all $x\in \Sigma(k_0,k_0')$.
\end{itemize}

For the case $1\in E$, one just applies Lemma~\ref{Lem:more-cutting} to have a new family $\Theta_g$ of affine reparametrizations such that $g\circ\sigma\circ\theta\circ\theta_g$ is $\varepsilon_g$-bounded. Define
$$\Gamma_1':=\{\gamma=\theta\circ\theta_g:~\theta\in\Theta,~\theta_g\in\Theta_g\}.$$
It is clear that
$$\#\Gamma_1'\le \#\Theta\cdot\#\Theta_g\le C_r{\rm e}^{\frac{k-k'}{r-1}}(\|Dg\|_{\rm sup}+2).$$

For the case $1\notin E$, for any $\theta\in \Theta$, there are two cases. Either $g\circ\sigma\circ\theta$ is $\varepsilon_g$-bounded, or $g\circ\sigma\circ\theta$ is not  $\varepsilon_g$-bounded.

\smallskip

If  $g\circ\sigma\circ\theta$ is not $\varepsilon_g$-bounded, then by Lemma~\ref{Lem:size-lower-bounde}, the length of $g\circ\sigma\circ\theta$ is larger than $2\beta_{\varepsilon_g}$ One defines two affine reparametrizations $\gamma_\theta^{-1}$ and $\gamma_\theta^{1}$ such that
\begin{itemize}
\item $\gamma_\theta^{-1}(-1)=-1$ and the length of $g\circ\sigma\circ\theta\circ\gamma_\theta^{-1}$ is $\beta_{\varepsilon_g}$.
\item $\gamma_\theta^{1}(1)=1$ and the length of $g\circ\sigma\circ\theta\circ\gamma_\theta^{1}$ is $\beta_{\varepsilon_g}$.
\end{itemize}
\begin{Claim}
For any $x\in \Sigma(k_0,k_0')$ and $g(x)\in g\circ\sigma\circ\theta([-1,1])$, one has 
$$x\in \sigma\circ \theta\circ \gamma_\theta^{-1}([-1,1])\cup \sigma\circ \theta\circ \gamma_\theta^{1}([-1,1]).$$
\end{Claim}
\begin{proof}[Proof of the claim]
Otherwise, one will have $g(x)\notin g\circ\sigma\circ \theta\circ \gamma_\theta^{-1}([-1,1])\cup g\circ\sigma\circ \theta\circ \gamma_\theta^{1}([-1,1])$. But by Lemma~\ref{Lem:size-lower-bounde}, the lengths of $g\circ\sigma\circ \theta\circ \gamma_\theta^{-1}([-1,1])$ and $g\circ\sigma\circ \theta\circ \gamma_\theta^{1}([-1,1])$ are both much larger than $\beta$ since $\beta$ can be chosen to be much smaller than $\beta_{\varepsilon_g}$. Note also the image of $g\circ\sigma\circ\theta$ is contained in the unstable manifold of $W^u(g(x))$. This implies that $g(x)\in L^{u,1}(\beta)$. Thus, we get a contradiction.
\end{proof}

Thus, we have the decomposition $\Theta=\Theta_L\cup \Theta_S$, where $\theta\in\Theta_L$ if and only if $g\circ\sigma\circ\theta$ is not $\varepsilon_g$-bounded. Define
$$\Theta_L^{\pm 1}=\{\theta\circ\gamma_\theta^{\pm 1}:~\theta\in\Theta_L\},$$
and
$$\Gamma_1'=\Theta_L^1\cup\Theta_L^{-1}\cup\Theta_S.$$
Thus for any $\gamma\in\Gamma_1'$, $g\circ\sigma\circ\gamma$ is $\varepsilon_g$-bounded. Indeed, we only have to consider the case of $\Theta_L^{\pm 1}$. Since the length is less than $2\beta_{\varepsilon_g}$, one knows the $\varepsilon_g$-bounded property by Lemma~\ref{Lem:size-lower-bounde}. From the Claim above, it is clear that
$$\Sigma(k_0,k_0')\subset\bigcup_{\gamma\in\Gamma_1'}\sigma\circ\gamma([-1,1]).$$
The cardinality of $\Gamma_1'$ can be estimated above by the case $1\in E$ or not. Thus, one gets the case of $j=1$.

\smallskip

Assume the case of $j$ is proved. This means one has a family $\Gamma_j'$ of reparametrizations such that the conclusion holds for $j$.  Now we prove the case of $j+1$. In fact it is close to the case $j=1$. For completeness, we give the proof. 

We consider the set $\Sigma(k_0,k_0',k_1,k_1',\cdots,k_{j-1},k_{j-1}',k_j,k_j')$. Take $\gamma'\in\Gamma_j'$. We know that $g^i\circ\sigma\circ\gamma'$ is $\varepsilon_g$-bounded for any $1\le i\le j$.

By applying the reparametrization Lemma (Lemma~\ref{Lem:reparametrization}), there is a family $\Theta_{\gamma'}$ of affine reparametrizations, whose cardinality satisfying
$$\#\Theta_{\gamma'}\le C_r{\rm e}^{\frac{k_j-k_j'}{r-1}},$$
such that
\begin{enumerate}
\item $\Sigma(k_0,k_0',k_1,k_1',\cdots,k_{j-1},k_{j-1}',k_j,k_j')\cap \sigma\circ\gamma'([-1,1])\subset \bigcup_{\theta\in\Theta_{\gamma'}}\sigma\circ\gamma'\circ \theta([-1,1])$
\item $g^{j+1}\circ \sigma\circ\gamma'\circ\theta$ is bounded;
\end{enumerate}
Note that $g^{j+1}\circ\sigma\circ\gamma'\circ\theta$ may not be $\varepsilon_g$-bounded. However, one notices that $g^j\circ\sigma\circ\gamma'\circ\theta$ is $\varepsilon_g$-bounded.
We have two cases
\begin{itemize}
\item $j+1\in E$: this means for all $x\in \Sigma(k_0,k_0',k_1,k_1',\cdots,k_{j-1},k_{j-1}',k_j,k_j')$, $g^{j+1}(x)=f^{(j+1)q}(x)\in L^{u,1}(\beta)$.
\item $j+1\notin E$: this means  for all $x\in \Sigma(k_0,k_0',k_1,k_1',\cdots,k_{j-1},k_{j-1}',k_j,k_j')$, $g^{j+1}(x)=f^{(j+1)q}(x)\notin L^{u,1}(\beta)$.
\end{itemize}

For the case $j+1\in E$, one just applies Lemma~\ref{Lem:more-cutting} to have a new family $\Theta_g$ of affine reparametrizations such that $g^j\circ\sigma\circ\gamma'\circ\theta\circ\theta_g$ is $\varepsilon_g$-bounded. Define
$$\Gamma_{j+1}':=\{\gamma=\gamma'\circ\theta\circ\theta_g:~\gamma'\in\Gamma_j',~\theta\in\Theta,~\theta_g\in\Theta_g\}.$$
It is clear that
\begin{align*}
\#\Gamma_{j+1}'&\le \#\Gamma_j'\cdot\#\Theta\cdot\#\Theta_g\le 2^j C_r^j\exp\{\sum_{i=0}^{j-1}\frac{k_i-k_i'}{r-1}\}(\|Dg\|_{\rm sup}+2)^{\#\{0< i\le j,~i\in E\}}\cdot C_r{\rm e}^{\frac{k_j-k_j'}{r-1}}(\|Dg\|_{\rm sup}+2)\\
&\le 2^{j+1} C_r^{j+1}\exp\{\sum_{i=0}^{j}\frac{k_i-k_i'}{r-1}\}(\|Dg\|_{\rm sup}+2)^{\#\{0< i\le j+1,~i\in E\}}
\end{align*}

For the case $j+1\notin E$, for any $\theta\in \Theta$, there are two cases. Either $g^{j+1}\circ\sigma\circ\gamma'\circ\theta$ is $\varepsilon_g$-bounded, or $g^{j+1}\circ\sigma\circ\gamma'\circ\theta$ is not $\varepsilon_g$-bounded.

\smallskip

If  $g^{j+1}\circ\sigma\circ\gamma'\circ\theta$ is not $\varepsilon_g$-bounded, one defines two affine reparametrizations $\gamma_\theta^{-1}$ and $\gamma_\theta^{1}$ such that
\begin{itemize}
\item $\gamma_\theta^{-1}(-1)=-1$ and the length $g^{j+1}\circ\sigma\circ\gamma'\circ\theta\circ\gamma_\theta^{-1}$ is $\beta_{\varepsilon_g}$.
\item $\gamma_\theta^{1}(1)=1$ and the length $g^{j+1}\circ\sigma\circ\gamma'\circ\theta\circ\gamma_\theta^{-1}$ is $\beta_{\varepsilon_g}$.
\end{itemize}
\begin{Claim}
For any $x\in \Sigma(k_0,k_0',k_1,k_1',\cdots,k_{j-1},k_{j-1}',k_j,k_j')$ and $g^{j+1}(x)\in g^{j+1}\circ\sigma\circ\gamma'\circ\theta([-1,1])$, one has 
$$x\in \sigma\circ\gamma'\circ\theta\circ \gamma_\theta^{-1}([-1,1])\cup  \sigma\circ\gamma'\circ\theta\circ \gamma_\theta^{1}([-1,1]).$$
\end{Claim}
\begin{proof}[Proof of the claim]The proof is similar to the claim as the case $j=1$.
Otherwise, one will have $g^{j+1}(x)\notin g^{j+1}\circ\sigma\circ \theta\circ \gamma_\theta^{-1}([-1,1])\cup g^{j+1}\circ\sigma\circ \theta\circ \gamma_\theta^{1}([-1,1])$. But by Lemma~\ref{Lem:size-lower-bounde}, the lengths of $g^{j+1}\circ\sigma\circ \theta\circ \gamma_\theta^{-1}([-1,1])$ and $g^{j+1}\circ\sigma\circ \theta\circ \gamma_\theta^{1}([-1,1])$ are both much larger than $\beta$ since $\beta$ can be chosen to be much smaller than $\beta_{\varepsilon_g}$. Note also the image of $g^{j+1}\circ\sigma\circ\gamma'\circ\theta$ is contained in the unstable manifold of $W^u(g^{j+1}(x))$ since the image $\sigma\circ\gamma'\circ\theta$ is contained in the unstable manifold of $x$. This implies that $g^{j+1}(x)\in L^{u,1}(\beta)$. Thus, we get a contradiction.
%
%
%
%
%
%
%
\end{proof}

Thus, we have the decomposition $\Theta=\Theta_L\cup \Theta_S$, where $\theta\in\Theta_L$ if and only if $g^{j+1}\circ\sigma\circ\gamma'\circ\theta$ is not $\varepsilon_g$-bounded. Define
$$\Theta_L^{\pm 1}=\{\gamma'\circ\theta\circ\gamma_\theta^{\pm 1}:~\theta\in\Theta_L\},$$
$$\Gamma_{j+1,\gamma'}'=\{\gamma'\circ \theta:~\theta\in\Theta_L^1\cup\Theta_L^{-1}\cup\Theta_S.\},$$
and 
$$\Gamma_{j+1}':=\bigcup_{\gamma'\in\Gamma_j'}\Gamma_{j+1,\gamma'}'.$$
Thus for any $\gamma\in\Gamma_{j+1}'$, $g^{j+1}\circ\sigma\circ\gamma$ is $\varepsilon_g$-bounded. From the Claim above, it is clear that
$$\Sigma(k_0,k_0',k_1,k_1',\cdots,k_{j-1},k_{j-1}',k_j,k_j')\subset\bigcup_{\gamma\in\Gamma_{j+1}'}\sigma\circ\gamma([-1,1]).$$
By the construction, we know that the cardinality of $\Gamma_{j+1}'$ is bounded by the quantity.
\end{proof}

\begin{Corollary}\label{Cor:cardinality-one-type}
For $n=mq$ large enough, for the family $\Gamma_m'$ as in Proposition~\ref{Pro:one-type-covering}, one has that
$$\#\Gamma_m'\le 2^m C_r^m\exp\{ n(\frac{1}{r}R(f)+\alpha/10)\}(\|Dg\|_{\rm sup}+2)^{\alpha_1 m+1},$$
\end{Corollary}
\begin{proof}
By Lemma~\ref{Lem:sum-of-k}, for $m$ large enough, one has 
$$\sum_{j=0}^{m-1}(k_j(y)-k_j'(y))\le n(\frac{1}{r}R(f)+\alpha/10)$$
By the choice of $E$, i.e., $d_m(E)<\alpha_1$, one has that 
$$\#\{0< i\le m,~i\in E\}<\alpha_1 m+1.$$
One can thus conclude by Proposition~\ref{Pro:one-type-covering}.
\end{proof}

\medskip

Now we prove Theorem~\ref{Thm:cover} for $n=mq$. We will prove that there is a family of reparametrizations $\widetilde\Gamma_{mq}$ such that the following holds:
\begin{enumerate}
\item $\lim_{n\to\infty}\frac{1}{n}\log\#\widetilde\Gamma_{mq}\le \frac{1}{r} R(f)+\alpha/2$;
\item $\bigcup_{\gamma\in\widetilde\Gamma_{mq}}\sigma\circ\gamma([-1,1])\supset K\cap \sigma_*$;
\item $g^j\circ\sigma\circ\gamma$ is strongly $\varepsilon_g$-bounded for any $\gamma\in\widetilde\Gamma_{mq}$ for any $j=0,1,\cdots,m-1$.
\end{enumerate}
Now we give the proof of the above statement.
\begin{proof}
For an admissible type $E\subset [0,m)$, for any admissible set
$$\Sigma(k_0,k_0',k_1,k_1',\cdots,k_{m-1},k_{m-1}')=\Sigma(E;k_0,k_0',k_1,k_1',\cdots,k_{m-1},k_{m-1}'),$$
one gets a family $\Gamma_j'=\Gamma_j'(\Sigma(k_0,k_0',k_1,k_1',\cdots,k_{m-1},k_{m-1}'))$ of reparametrizations by applying Proposition~\ref{Pro:one-type-covering}. Thus, to take the union of all these reparametrizations, we get a family of reparametrizations $\widetilde\Gamma_{mq}$. By the property of $\Gamma_j'$, one has that $g^j\circ\sigma\circ\gamma$ is strongly $\varepsilon_g$-bounded for any $\gamma\in\widetilde\Gamma_{mq}$ for any $j=0,1,\cdots,m-1$. Clearly, one has 
$$\bigcup_{\gamma\in\widetilde\Gamma_{mq}}\sigma\circ\gamma([-1,1])\supset K\cap \sigma_*.$$
It remains to estimate the cardinality of $\widetilde\Gamma_{mq}$.
\begin{align*}
&\#\widetilde\Gamma_{mq}\le \#{\mathcal S}_m\times \sup_{E\in {\mathcal S}_m}\#\{(k_0,k_0',k_1,k_1',\cdots,k_{m-1},k_{m-1}'):~\Sigma(E;k_0,k_0',k_1,k_1',\cdots,k_{m-1},k_{m-1}')\neq\emptyset\}\\
&\times \sup_{E\in {\mathcal S}_m;(k_0,k_0',k_1,k_1',\cdots,k_{m-1},k_{m-1}')}\# \Gamma_j'(\Sigma(k_0,k_0',k_1,k_1',\cdots,k_{m-1},k_{m-1}')).
\end{align*}
Thus, by Equation~\eqref{e.E-type}, Lemma~\ref{Lem:derivetive-type}, Proposition~\ref{Pro:one-type-covering} and Corollary~\ref{Cor:cardinality-one-type}, one has that
\begin{align*}
&\#\widetilde\Gamma_{mq}\le {\rm e}^{\alpha n/10}\times \big(q\cdot \max\{\log\|Df\|_{\rm sup},~\log\|Df^{-1}\|_{\rm sup}\}\big)^m\\
&~~~~~~~~~~~\times 2^m C_r^m\exp\{ n(\frac{1}{r}R(f)+\alpha/10)\}(\|Dg\|_{\rm sup}+2)^{\alpha_1 m+1}.
\end{align*}
By taking $\log$ and dividing $n=mq$, one has that
\begin{align*}
&\frac{1}{n}\log\#\widetilde\Gamma_{mq}\le \alpha/10+ \frac{1}{q}\log q+\frac{1}{q} \log\big(\max\{\log\|Df\|_{\rm sup},~\log\|Df^{-1}\|_{\rm sup}\}\big)\\
&~~~~~~~~~~~+\frac{1}{q}\log(2 C_r)+\frac{1}{r}R(f)+\alpha/10 +\frac{\alpha_1 m+1}{mq}\log(\|Dg\|_{\rm sup}+2)\\
&~~~~~~~~~~~\le \frac{1}{r} R(f)+\alpha/2
\end{align*}
by Equation~\eqref{e.q-delete-Cr}, Equation~\eqref{q-delete-logq}, Equation~\eqref{e.q-delete-norm}.
\end{proof}

\medskip

Now we choose the reparametrization $\Gamma_n$ in Theorem~\ref{Thm:cover} for any $n\in\NN$. Write $n=mq+\ell$ with $0\le \ell<q$. Define
$$\Gamma_n=\widetilde\Gamma_{mq}.$$
To prove Theorem~\ref{Thm:cover}, it remains to prove for any $n$.
$$\|D(f^j\circ\sigma\circ\gamma)\|\le 1,~~\forall j=0,1,\cdots,n-1.$$
We have known that for any $m$, for any $\gamma\in\Gamma_{mq}$
$$\|D(f^{qm}\circ\sigma\circ\gamma)\|\le \varepsilon_g,~~\forall j=0,1,\cdots,m-1.$$
Thus, for $0\le \ell<q$, by the choice of $\varepsilon_g$ (Equation~\ref{e.reduce-epsilon-g}), one has that
$$\|D(f^{qm+\ell}\circ\sigma\circ\gamma)\|\le\|Df\|_{\rm sup}^q\varepsilon_g\le 1.$$
The proof of Theorem~\ref{Thm:cover} is thus complete. \qed

\appendix

\section{Measurability}\label{Sec:measurability}
Given $\chi_1>\chi_2$ and $k\in \NN$, we say a point $x\in M$  has $(k,\chi_1,\chi_2)$-\textit{dominated Lyapunov exponents} if there exists a unique splitting $T_xM=E(x)\bigoplus F(x)$ such that $\dim E(x)=k$ and
\begin{itemize}
\item for every $v_E\in E(x)\setminus \{0\}$ and $v_F\in F(x)\setminus \{0\}$
$$\lim_{n\rightarrow \pm\infty } \frac{1}{n}\log \|D_xf (v_E)\|>\chi_1>\chi_2>\lim_{n\rightarrow \pm\infty } \frac{1}{n}\log \|D_xf (v_F)\|.$$
\item $$ \lim_{n\rightarrow \pm\infty }\frac{1}{n} \log \sin\angle(E(f^n(x)),F(f^n(x)))=0.$$
\end{itemize}
Denote by ${\rm DOM}_{\chi_1,\chi_2}^k$ the set of all points with $(k,\chi_1,\chi_2)$-dominated Lyapunov exponents.
It follows that ${\rm DOM}_{\chi_1,\chi_2}^k$ is measurable and $f$-invariant.

An $f$-invariant measure $\mu$ is said to have $(k,\chi_1,\chi_2)$\textit{-dominated Lyapunov exponents} if $\mu ({\rm DOM}_{\chi_1,\chi_2}^k)=1$.
Choose $\varepsilon>0$ small enough such that 
$$0<\varepsilon<\min \left\{\frac{|\chi_1|}{10\dim M}, \frac{\chi_1-\chi_2}{20\dim M}, \frac{|\chi_2|}{10\dim M}\right\}.$$
By Pesin's nonuniformly hyperbolic theory \cite{BP07}, there exists an $f$-invariant measurable subset $\Lambda^{k,\varepsilon}_{\chi_1,\chi_2}\subset {\rm DOM}_{\chi_1,\chi_2}^k$, a measurable function $C:\Lambda^{k,\varepsilon}_{\chi_1,\chi_2}\rightarrow (0,+\infty)$ and a measurable splitting $T_{\Lambda^{k,\varepsilon}_{\chi_1,\chi_2}}=E\oplus F$ satisfying 
\begin{enumerate}
	\item [(1)] $\mu(\Lambda^{k,\varepsilon}_{\chi_1,\chi_2})=1$ for any $f$-invariant measure $\mu$ with $(k,\chi_1,\chi_2)$\textit{-dominated Lyapunov exponents};
	\item [(2)] $e^{-\varepsilon}<C(f(x))/C(x)<e^{\varepsilon}$ for any $x\in \Lambda^{k,\varepsilon}_{\chi_1,\chi_2}$;
	\item [(3)] $\dim E(x)=k$ and $\angle(E(x),F(x))>C(x)^{-1}$ for any $x\in \Lambda^{k,\varepsilon}_{\chi_1,\chi_2}$;
	\item [(4)] for every $x\in \Lambda^{k,\varepsilon}_{\chi_1,\chi_2}$, for all unit vectors $v_E\in E(x),v_F\in F(x)$ and $n\in \NN $
	\begin{align*}
	&\|D_xf^{-n}v_E\|\leq C(x)e^{-(\chi_1-\varepsilon)n}, \ \|D_xf^{n}v_F\|\leq C(x)e^{(\chi_2+\varepsilon)n};\\
	&\|D_xf^{n}v_E\|\geq C(x)^{-1}e^{(\chi_1-\varepsilon)n}, \ \|D_xf^{-n}v_F\|\geq C(x)^{-1}e^{-(\chi_2+\varepsilon)n}.
	\end{align*}   
\end{enumerate}
For $\ell >1$, we define
$$\Lambda^{k,\varepsilon}_{\chi_1,\chi_2,\ell}:=\{x\in \Lambda^{k,\varepsilon}_{\chi_1,\chi_2}:~C(x)\leq e^{\ell \varepsilon}\}.$$
Then, $\Lambda^{k,\varepsilon}_{\chi_1,\chi_2,\ell}$ is closed,  $\bigcup_{\ell>1}\Lambda^{k,\varepsilon}_{\chi_1,\chi_2,\ell}=\Lambda^{k,\varepsilon}_{\chi_1,\chi_2}$,
$f^{n}(\Lambda^{k,\varepsilon}_{\chi_1,\chi_2,\ell})\subset  \Lambda^{k,\varepsilon}_{\chi_1,\chi_2,\ell+n}$ for any $n\in\ZZ$, and $x\mapsto E(x), x\mapsto F(x)$ depend continuously on $\Lambda^{k,\varepsilon}_{\chi_1,\chi_2,\ell}$
for every $n\in \NN$ and $\ell >1$.

Assuming $\chi_1>0$,  by the local unstable manifold theory (see \cite[Section 7.1]{BP07}), for any $x\in \Lambda^{k,\varepsilon}_{\chi_1,\chi_2}$, the $k$-local unstable manifold $W^{u,k}_{\loc}(x)$ is a $C^{r}$ embedding sub-manifold, 
$fW^{u,k}_{\loc}(x)\supset W^u_{\loc}(fx)$
and $x\mapsto W^{u,k}_{\loc}(x)$ is continuous on $\Lambda^{k,\varepsilon}_{\chi_1,\chi_2,\ell}$ for any $\ell>1$.
The global $k$-unstable manifold $W^{u,k}(x)$ satisfies
$$W^{u,k}(x)=\bigcup_{n\ge 0} f^{n} W^{u,k}_{\loc}(f^{-n}x).$$

For $\chi_1>\chi_2$, fix some $\varepsilon(\chi_1,\chi_2)$ small enough, let
$\Lambda^{k}_{\chi_1,\chi_2}:=\Lambda^{k,\varepsilon(\chi_1,\chi_2)}_{\chi_1,\chi_2}$ and define  $\Lambda^{k}_{\chi_1,\chi_2,\ell }$ similarly.
Define
$$\Lambda^{k}:=\bigcup \{\Lambda^{k}_{\chi_1,\chi_2}:~\chi_1>\chi_2, \ \chi_1,\chi_2\in \QQ\},$$
and 
$$\Lambda^{k,u}:=\bigcup \{\Lambda^{k}_{\chi_1,\chi_2}:~\chi_1>\max \{\chi_2,0\}, \ \chi_1,\chi_2\in \QQ\}.$$
We know that $\Lambda^{k}$ and $\Lambda^{k,u}$ are measurable sets.
For any ergodic measure $\mu$ has one-dominated positive Lyapunov exponent we have $\mu(\Lambda^{1,u})=1$.

Recall the sets $L^{u,1}(\beta)$ defined in Section \ref{SEC:1}.
We now prove the measurability of $L^{u,1}(\beta)$.
\begin{Proposition}\label{Prop:A1}
	The set $L^{u,1}(\beta)$ is measurable for any $\beta>0$.  
\end{Proposition}

\begin{proof}
    It suffices to show that for any $\chi_1,\chi_2\in \QQ$ and $\chi_1>\max\{\chi_2,0\}$, the set 
    $L^{u,1}(\beta)\cap \Lambda^{1}_{\chi_1,\chi_2}$ is measurable.
	
	For any $x\in \Lambda^{1}_{\chi_1,\chi_2}$, we have
	$$ W^{u,1}(x):=\bigcup_{n\geq 1} f^n W^{u,1}_{\rm loc}(f^{-n} x), $$
	and the sequence $\left\{f^n W^{u,1}_{\rm loc}(f^{-n} x)\right\}_{n\in \NN}$ is increasing in $n$. 
	For each $n>0$, define
	\begin{align*}
		~L^{u,1}_n(\beta)=&\{x\in \Lambda^{1}_{\chi_1,\chi_2}:~\exists W_x\subset f^n W^{u,1}_{\rm loc}(f^{-n} x),~\textrm{s.t.}~\exp_x^{-1}W_x~\textrm{is a $C^1$ graph of a map~} \\ &~\varphi:~E(x)\to E(x)^\perp,
		~{\rm Lip}(\varphi)\le 1/3,~{\rm Domain}(\varphi)\supset E(x)(\beta)\}.
	\end{align*}
	Then, we have 
\begin{equation}\label{e.limit-measurable} 
L^{u,1}(\beta)\cap \Lambda^{1}_{\chi_1,\chi_2}=\bigcup_{n\geq 1,\ell \geq 1} L^{u,1}_n(\beta)\cap \Lambda^{1}_{\chi_1,\chi_2,\ell}.
\end{equation}

	\begin{Claim}
		For any $n \geq 1$ and $\ell \geq 1$, the set $L^{u,1}_n(\beta)\cap \Lambda^{1}_{\chi_1,\chi_2,\ell}$ is closed.
	\end{Claim}
\begin{proof}[Proof of the Claim]	
	Let $(x_m)_{m \geq 0}$ be a sequence of points in $L^{u,1}_n(\beta)\cap \Lambda^{1}_{\chi_1,\chi_2,\ell}$ with $x_m \rightarrow x$.     
	Since $\Lambda^{1}_{\chi_1,\chi_2,\ell}$ is closed, $x\in \Lambda^{1}_{\chi_1,\chi_2,\ell}$. 
	Note that $f^{-n}(x_m)\in \Lambda^{1}_{\chi_1,\chi_2,\ell+n}$ and $f^{-n}(x_m)\rightarrow f^{-n}(x)$.
	Thus, $W^{u,1}_{\rm loc}(f^{-n}(x_m))$ converges to $W^{u,1}_{\rm loc}(f^{-n}(x))$ as $m \rightarrow +\infty$ in the $C^1$ topology.
	Therefore, we also have $f^n W^{u,1}_{\rm loc}(f^{-n}(x_m))$  converges to $f^n W^{u,1}_{\rm loc}(f^{-n}(x))$ as $m \rightarrow +\infty$ in the $C^1$ topology.
		
	Note that $E(x_m)$ also converges to $E(x)$
	and $\{x_m\}_{m\in \NN} \subset L^u_n(\beta)$, it follows that there exists a map $\varphi:~E(x)\to E(x)^\perp$ with ${\rm Domain}(\phi)\supset E(x)(\beta)$ and ${\rm Lip}(\varphi)\leq 1/3$, such that $\exp_x {\rm graph}(\varphi) \subset f^n W^{u,1}_{\rm loc}(f^{-n}(x))$. Thus, $x\in L^{u,1}_n(\beta)$ and so $L^{u,1}_n(\beta)\cap \Lambda^{1}_{\chi_1,\chi_2,\ell}$ is closed.
\end{proof}
The proof of the proposition follows immediately from the claim and Equation~\eqref{e.limit-measurable}.
\end{proof}
 
 \subsection*{Data availibility statement}
 No data availability statement is required, as no experimental data is involved.
 
  \subsection*{Ethics declarations}
\noindent \textbf{Conflict of interest} \
  On behalf of all authors, the corresponding author states that there is no conflict of interest.

\vskip 5pt

\flushleft{\bf Chiyi Luo} \\
\small School of Mathematics and Statistics, Jiangxi Normal University, Nanchang,   330022, P. R. China\\
\textit{E-mail:} \texttt{luochiyi98@gmail.com}\\

\flushleft{\bf Dawei Yang} \\
\small School of Mathematical Sciences,  Soochow University, Suzhou, 215006, P.R. China\\
\textit{E-mail:} \texttt{yangdw@suda.edu.cn}\\

\end{document}